\documentclass[12pt,reqno,a4paper]{amsart}
\usepackage{hyperref}
\usepackage{enumerate}
\usepackage[left=0.7in, right=0.7in, top= 1in, bottom=1.75in]{geometry}
\usepackage{amsmath,amssymb,amsthm,amsfonts}
\usepackage{graphicx}
\usepackage{tikz}
\usetikzlibrary{graphs, graphs.standard}
\usetikzlibrary{shapes}
\usetikzlibrary{plotmarks}
\usetikzlibrary{arrows}
\usetikzlibrary{positioning}
\theoremstyle{definition}
\newtheorem{defn}{Definition}[section]
\newtheorem{fact}[defn]{Fact}
\newtheorem{prop}[defn]{Proposition}
\newtheorem{thm}[defn]{Theorem}

\newtheorem{question}[defn]{Question}
\newtheorem{remark}[defn]{Remark}

\newtheorem{claim}[defn]{Claim}
\title[Distinguishing colorings, proper colorings, and covering properties without AC]{Distinguishing colorings, proper colorings, and covering properties without the Axiom of Choice}

\author{Amitayu Banerjee}
\address{Alfr\'ed R\'enyi Institute of Mathematics, Reáltanoda utca 13-15, 1053, Budapest, Hungary}
\email[Corresponding author]{banerjee.amitayu@gmail.com}

\author{Zal\'{a}n Moln\'{a}r}
\address{E\"otv\"os Lor\'and University, Department of Logic, M\'{u}zeum krt. 4, 1088, Budapest, Hungary}
\email{mozaag@gmail.com}

\author{Alexa Gopaulsingh}
\address{E\"otv\"os Lor\'and University, Department of Logic, M\'{u}zeum krt. 4, 1088, Budapest, Hungary}
\email{alexa279e@gmail.com}

\date{}

\makeatletter
\@namedef{subjclassname@2020}{\textup{2020} Mathematics Subject Classification}
\makeatother
\begin{document}
	\begin{abstract}
		We work with simple graphs in $\mathsf{ZF}$ (i.e., the Zermelo–Fraenkel set theory without the Axiom of Choice ($\mathsf{AC}$)) and 
		assume that the 
		sets of colors can be either well-orderable or non-well-orderable,
		to prove that the following statements are equivalent to K\H{o}nig’s Lemma:
		\begin{enumerate}[(a)]
			\item Any infinite locally finite connected graph $G$ such that the minimum degree of $G$ is greater than $k$, has a chromatic number for any fixed integer $k$ greater than or equal to 2.
			\item Any infinite locally finite connected graph has a chromatic index.
			\item Any infinite locally finite connected graph has a distinguishing number.
			\item Any infinite locally finite connected graph has a distinguishing index.
		\end{enumerate}
		The above results strengthen some recent results of Stawiski since he assumed that the 
		sets of colors can be well-ordered. 
		We also formulate new conditions for the existence of irreducible proper coloring, minimal edge cover, maximal matching, and minimal dominating set in connected bipartite graphs and locally finite connected graphs, which are either equivalent to $\mathsf{AC}$ or K\H{o}nig’s Lemma. Moreover, we show that if the Axiom of Choice for families of 2-element sets holds, then the Shelah-Soifer graph has a minimal dominating set.
	\end{abstract}
	\maketitle
	\section{Introduction}
	In 1991, Galvin--Komj\'{a}th proved that the statements ``Any graph has a chromatic number" and ``Any graph has an irreducible proper coloring" are equivalent to $\mathsf{AC}$ in $\mathsf{ZF}$ using Hartogs's theorem (cf. \cite{GK1991}). 
	In 1977, Babai \cite{Bab1977} introduced distinguishing vertex colorings under the name asymmetric colorings, and distinguishing edge colorings were introduced by Kalinowski--Pil\'{s}niak \cite{KP2015} in 2015.
	Recently, Stawiski \cite{Sta2023} proved that the statements (b)--(d) mentioned in the abstract above and the statement ``Any infinite locally finite connected graph has a chromatic number" are equivalent to K\H{o}nig's Lemma (a weak form of $\mathsf{AC}$) by assuming that the 
	sets of colors can be well-ordered (cf. \cite[Lemma 3.3 and section 2]{Sta2023}).
	
	
	\subsection{Proper and distinguishing colorings}
	An infinite cardinal in $\mathsf{ZF}$ can either be an ordinal or a set that is not well-orderable. 
	Herrlich--Tachtsis \cite[Proposition 23]{HT2006} proved that no Russell graph has a chromatic number in $\mathsf{ZF}$. We refer the reader to
	\cite{HT2006} for the details concerning Russell graph and Russell sequence.
	In Theorem 4.2, the first and the second authors study new combinatorial proofs (mainly inspired by the arguments of \cite[Proposition 23]{HT2006}) to show that the statements (a)--(d) mentioned in the abstract above are equivalent to K\H{o}nig's Lemma (without assuming that the sets of colors can be well-ordered).\footnote{We note that statement (a) mentioned in the abstract is a new equivalent of K\H{o}nig's Lemma. Stawiski's graph from \cite[Theorem 3.6]{Sta2023} shows that K\H{o}nig's Lemma is equivalent to ``Every infinite locally finite connected graph $G$ such that $\delta(G)$ (the minimum degree of $G$) is 2 has a chromatic number".}
	
	\subsection{New equivalents of K\H{o}nig's lemma and AC} 
	The role of $\mathsf{AC}$ and K\H{o}nig’s Lemma in the existence of graph-theoretic properties like irreducible proper coloring, chromatic numbers, maximal independent sets, spanning trees, and distinguishing colorings were studied by several authors in the past (cf. \cite{Ban2023, Ban, DM2006, Fri2011, HH1973, GK1991, Spa2014, Sta2023}). We list a few known results apart from the above-mentioned results due to Galvin--Komj\'{a}th \cite{GK1991} and Stawiski \cite{Sta2023}. In particular, Friedman \cite[Theorem 6.3.2, Theorem 2.4]{Fri2011} proved
	that $\mathsf{AC}$ is equivalent to the statement ``Any graph has a maximal independent set". 
	H\"{o}ft--Howard \cite{HH1973} proved that the statement ``Any connected graph contains a partial subgraph which is a tree" is equivalent to $\mathsf{AC}$. Fix any even integer $m \geq 4$ and any integer $n\geq 2$. Delhomm\'{e}--Morillon \cite{DM2006} studied the role of $\mathsf{AC}$ in the existence of spanning subgraphs and observed that $\mathsf{AC}$ is equivalent to ``Any connected bipartite graph has a spanning subgraph without a complete bipartite subgraph $K_{n,n}$" as well as ``Any connected graph admits a spanning $m$-bush" (cf. \cite[Corollary 1, Remark 1]{DM2006}). They also proved that the statement ``Any locally finite connected graph has a spanning tree" is equivalent to K\H{o}nig's lemma in \cite[Theorem 2]{DM2006}. Banerjee \cite{Ban2023, Ban} observed that the statements ``Any infinite locally finite connected graph has a maximal independent set" and ``Any infinite locally finite connected graph has a spanning $m$-bush" are equivalent to K\H{o}nig's lemma. 
	However, the existence of maximal matching, minimal edge cover, and minimal dominating set in $\mathsf{ZF}$ were not previously investigated. The following table summarizes the new results (cf. Theorem 5.1, Theorem 6.4).\footnote{We note that Theorem 5.1 is a combined effort of the first and the second authors. Moreover, all remarks in Section 6 including Theorem 6.4 are due to all the authors.}
	\begin{center}
		\begin{tabular}{|l|l|l| } 
			\hline New equivalents of K\H{o}nig's lemma & New equivalents of $\mathsf{AC}$ \\ 
			\hline
			$\mathcal{P}_{lf,c}$(irreducible proper coloring) (Theorem 5.1) & \\ 
			$\mathcal{P}_{lf,c}$(minimal dominating set) (Theorem 5.1) & $\mathcal{P}_{c}$(minimal dominating set) (Theorem 6.4)
			\\  
			$\mathcal{P}_{lf,c}$(maximal matching) (Theorem 5.1) & $\mathcal{P}_{c,b}$(maximal matching) (Theorem 6.4) 
			\\
			$\mathcal{P}_{lf,c}$(minimal edge cover) (Theorem 5.1) & $\mathcal{P}_{c,b}$(minimal edge cover) (Theorem 6.4) \\
			\hline
		\end{tabular}
	\end{center}
	
	In the table, $\mathcal{P}_{lf,c}$(property $X$) denotes ``Any infinite locally finite connected graph has property $X$", $\mathcal{P}_{c,b}$(property $X$)
	denotes ``Any connected bipartite graph has property $X$" and $\mathcal{P}_{c}$(property $X$)
	denotes ``Any connected graph has property $X$".

	\section{Basics}
	\begin{defn}
		Suppose $X$ and $Y$ are two sets. We write:
		
		\begin{enumerate}
			\item $X \preceq Y$, if there is an injection $f : X \rightarrow Y$.
			\item $X$ and $Y$ are equipotent if $X \preceq Y$ and $Y \preceq X$, i.e., if there is a bijection $f : X \rightarrow Y$.
			\item $X \prec Y$, if $X \preceq Y$ and $X$ is not equipotent with $Y$.
		\end{enumerate}
	\end{defn}
	
	\begin{defn}
		Without $\mathsf{AC}$, a set $m$ is called a {\em cardinal} if it is the cardinality $\vert x\vert$ of some set $x$, where 
		$\vert x\vert$ = $\{y : y \sim x$ and $y$ is of least rank$\}$ where $y \sim x$ means the
		existence of a bijection $f : y \rightarrow x$ 
		(see \cite[Definition 2.2, p. 83]{Lev2002} and \cite[Section 11.2]{Jec1973}).
	\end{defn}
	
	\begin{defn}
		A graph $G=(V_{G}, E_{G})$ consists of a set $V_{G}$ of vertices and a set $E_{G}\subseteq [V_{G}]^{2}$ of edges.\footnote{i.e., $E_{G}$ is a subset of the set of all two-element subsets of $V_{G}$.} Two vertices $x, y \in V_{G}$ are {\em adjacent vertices} if $\{x, y\} \in E_{G}$, and two edges $e,f\in E_{G}$ are {\em adjacent edges} if they share a common vertex. 
		The {\em degree} of a vertex $v\in V_{G}$, denoted by $deg(v)$, is the number of edges emerging from $v$. We denote by $\delta(G)$ the minimum degree of $G$. Given a non-negative integer $n$, a {\em path of length $n$} in $G$ is a one-to-one finite sequence $\{x_{i}\}_{0\leq i \leq n}$ of vertices such that for each $i < n$, $\{x_{i}, x_{i+1}\} \in E_{G}$; such a path joins $x_{0}$ to $x_{n}$. 
		\begin{enumerate}
			\item $G$ is {\em locally finite} if every vertex of $G$ has a finite degree.
			\item $G$ is {\em connected} if any two vertices are joined by a path of finite length.
			\item A {\em dominating set} of $G$ is a set $D$ of vertices of $G$, such that any vertex of $G$ is either in $D$, or has a neighbor in $D$. 
			
			\item An {\em independent set} of $G$
			is a set of vertices of $G$, no two of which are adjacent vertices. A {\em dependent set} of $G$ is a set of vertices of $G$ that is not an independent set.
			
			\item  A {\em vertex cover} of $G$ is a set of vertices of $G$ that includes at least one endpoint of every edge of the graph $G$.
			
			\item A {\em matching} $M$ in $G$ is a set of pairwise non-adjacent edges. 
			
			\item An {\em edge cover} of $G$ is a set $C$ of edges such that each vertex in $G$ is incident with at least one edge in $C$.
			
			\item   A {\em minimal dominating set (minimal vertex cover, minimal edge cover)} is a dominating set (a vertex cover, an edge cover) that is not a superset of any other dominating set (vertex cover, edge cover).  A {\em maximal independent set (maximal matching)} is an independent set (a matching) that is not a subset of any other independent set (matching).
			
			\item A {\em proper vertex coloring} of  $G$ with a color set $C$ is a mapping $f:V_{G}\rightarrow C$ such that for every $\{x,y\}\in E_{G}$, $f(x)\not= f(y)$. A {\em proper edge coloring} of $G$ with a color set $C$ is a mapping $f:E_{G}\rightarrow C$ such that for any two adjacent edges $e_{1}$ and $e_{2}$, $f(e_{1})\not= f(e_{2})$.

			\item Let $\vert C\vert =\kappa$. We say $G$ is {\em $\kappa$-proper vertex colorable} or {\em $C$-proper vertex colorable} if there is a proper vertex coloring $f: V_{G} \rightarrow C$ and $G$ is {\em $\kappa$-proper edge colorable} or {\em $C$-proper edge colorable} if there is a proper edge coloring $f: E_{G} \rightarrow C$. The least cardinal $\kappa$ for which $G$ is $\kappa$-proper vertex colorable (if it exists) is the {\em chromatic number} of $G$ and the least cardinal $\kappa$  for which $G$ is $\kappa$-proper edge colorable (if it exists) is the {\em chromatic index} of $G$. 
			
			\item A proper vertex coloring $f: V_{G} \rightarrow C$ is a {\em $C$-irreducible proper coloring} if $f^{-1}(c_{1})\cup f^{-1}(c_{2})$ is a dependent set whenever $c_{1}, c_{2}\in C$ and $c_{1} \not= c_{2}$ (cf. \cite{GK1991}). 
			
			\item An automorphism of $G$ is a bijection $\phi: V_{G}\rightarrow V_{G}$ such that $\{u,v\}\in E_{G}$ if and only if $\{\phi(u),\phi(v)\}\in E_{G}$. 
			Let $f$ be an assignment of colors to either vertices or edges of $G$. 
			We say that an automorphism $\phi$ of $G$ {\em preserves $f$} if each vertex of $G$ is mapped to a vertex of the same color or each edge of $G$ is mapped to an edge of the same color.
			We say that $f$ is a {\em distinguishing coloring} if the only automorphism that preserves $f$ is the identity. Let $\vert C\vert =\kappa$. We say $G$ is {\em $\kappa$-distinguishing vertex colorable} or {\em $C$-distinguishing vertex colorable} if there is a distinguishing vertex coloring $f: V_{G} \rightarrow C$ and $G$ is {\em $\kappa$-distinguishing edge colorable} or {\em $C$-distinguishing edge colorable} if there is a distinguishing edge coloring $f: E_{G} \rightarrow C$. The least cardinal $\kappa$ for which $G$ is $\kappa$-distinguishing vertex colorable (if it exists) is the {\em distinguishing number} of $G$ and the least cardinal $\kappa$ for which $G$ is $\kappa$-distinguishing edge colorable (if it exists) is the {\em distinguishing index} of $G$. 
			
			\item The automorphism group of $G$, denoted by $Aut(G)$, is the group consisting of automorphisms of $G$ with composition as the operation. Let $\tau$ be a group acting on a set $S$ and let $a\in S$. The orbit of $a$, denoted by $Orb_{\tau}(a)$, is the set $\{\phi(a) : \phi \in \tau\}$.
			
			\item $G$ is {\em complete} if each pair of vertices is connected by an edge. We denote by $K_{n}$, the complete graph on $n$ vertices for any natural number $n\geq 1$.
			
			\item {\em K\H{o}nig’s Lemma} states that every infinite locally finite connected graph has a ray.
		\end{enumerate}
	\end{defn}
	
	Let $\omega$ be the set of natural numbers, $\mathbb{Z}$ be the set of integers, $\mathbb{Q}$ be the set of rational numbers, $\mathbb{R}$ be the set of real numbers, and $\mathbb{Q}+a = \{a + r: r \in \mathbb{Q}\}$ for any $a \in \mathbb{R}$.
	Shelah--Soifer \cite{SS2003} constructed a graph whose chromatic number is 2 in $\mathsf{ZFC}$ and uncountable in some model of $\mathsf{ZF}$ (e.g. in Solovay's model from \cite[Theorem 1]{Sol1970}).
	
	\begin{defn}{(cf. \cite{SS2003})} The {\em Shelah--Soifer Graph} $G = (\mathbb{R},\rho)$ is defined by
		$x\rho y \Leftrightarrow (x-y) \in (\mathbb{Q}+\sqrt{2}) \cup (\mathbb{Q}+ (-\sqrt{2}))$.    
	\end{defn}

	\begin{defn}
		A set $X$ is {\em Dedekind-finite} if it satisfies the following equivalent conditions (cf. \cite[Definition 1]{HT2006}):
		\begin{itemize}
			\item $\omega\not\preceq X$,\footnote{i.e., there is no injection $f:\omega\rightarrow X$.}
			\item  $A \prec X$ for every proper subset $A$ of $X$.
		\end{itemize}
	\end{defn}
	
	\begin{defn}
		For every family $\mathcal{B}=\{B_{i}:i\in I\}$ of non-empty sets, $\mathcal{B}$ is said to have a {\em partial choice function} if  $\mathcal{B}$ has an infinite
		subfamily $\mathcal{C}$ with a choice function.
	\end{defn}
	
	\begin{defn}{(A list of choice forms).}
		
		\begin{enumerate}
			\item $\mathsf{AC_{2}}$: Every family of 2-element sets has a choice function.
			\item $\mathsf{AC_{fin}}$: Every family of non-empty finite sets has a choice function. 
			
			\item $\mathsf{AC_{fin}^{\omega}}$: Every countably infinite family of non-empty finite sets has a choice function. We recall that $\mathsf{AC_{fin}^{\omega}}$ is equivalent to
			K\H{o}nig’s Lemma as well as the statement ``The union of a countable
			family of finite sets is countable".
			\item $\mathsf{AC^{\omega}_{\text{$k$}\times fin}}$ for $k\in \omega\backslash\{0,1\}$: Every countably infinite family $\mathcal{A}=\{A_{i}:i\in\omega\}$ of non-empty finite sets, where $k$ divides $\vert A_{i} \vert$, has a choice function.
			
			\item $\mathsf{PAC^{\omega}_{\text{$k$}\times fin}}$ for $k\in \omega\backslash\{0,1\}$: Every countably infinite family $\mathcal{A}=\{A_{i}:i\in\omega\}$ of non-empty finite sets, where $k$ divides $\vert A_{i} \vert$ has a partial choice function. 
			
		\end{enumerate}
	\end{defn}
	\begin{defn}
		From the point of view of model theory, the {\em language of graphs $\mathcal{L}$} consists of a single binary relational symbol $E$ depicting edges, i.e., $\mathcal{L}=\{E\}$ and a graph is an $\mathcal{L}$-structure $G=\langle V, E\rangle$ consisting of a non-empty set $V$ of vertices and the edge relation $E$ on $V$. 
		Let $G=\langle V, E\rangle$ be an $\mathcal{L}$-structure, $\phi(x_{1},..., x_{n})$ be a first-order $\mathcal{L}$-formula, and let $a_1,...,a_n \in V$ for some $n\in\omega\backslash\{0\}$.
		We write
		$G \models \phi(a_1,...,a_n)$,
		if the property expressed by $\phi$ is true in $G$ for $a_1,...,a_n$.
		Let $G_{1}=\langle V_{G_{1}}, E_{G_{1}}\rangle$ and $G_{2}=\langle V_{G_{2}}, E_{G_{2}}\rangle$ be two $\mathcal{L}$-structures. We recall that if $j: V_{G_{1}}\to V_{G_{2}}$ is an isomorphism, $\varphi(x_{1},..., x_{r})$ is a first-order $\mathcal{L}$-formula on $r$ variables for some $r\in\omega\backslash\{0\}$, and $a_{i}\in V_{G_{1}}$ for each $1\leq i\leq r$, then by induction on the complexity of formulae, one can see that $G_{1}\models\varphi(a_{1},..., a_{r})$ if and only if $G_{2}\models\varphi(j(a_{1}),..., j(a_{r}))$ (cf. \cite[Theorem 1.1.10]{Mar2002}). 
	\end{defn}
	
	\section{Known and Basic Results}
	\subsection{Known Results}
	\begin{fact} {($\mathsf{ZF}$)}
		{\em The following hold:
			\begin{enumerate}
				\item (Galvin--Komj\'{a}th; cf. \cite[Lemma 3 and the proof of Lemma 2]{GK1991}) Any graph based on a well-ordered set of vertices has an irreducible proper coloring and a chromatic number.
				
				\item {(Delhomm\'{e}--Morillon; cf. \cite[Lemma 1]{DM2006})
					Given a set $X$ and a set $A$ which is the range of no mapping with domain $X$, consider a mapping $f: A \rightarrow \mathcal{P}(X)\backslash \{\emptyset\}$ (with values non-empty subsets of $X$). Then there are distinct $a$ and $b$ in $A$ such that $f(a)\cap f(b) \neq \emptyset$.
					
					\item {(Herrlich--Rhineghost; cf. \cite[Theorem]{HR2005})} For any measurable subset $X$ of $\mathbb{R}$ with a positive measure there exist $x \in X$ and $y \in X$ with $y-x \in \mathbb{Q}+ \sqrt{2}$.
					\item (Stawiski; cf. \cite[proof of Theorem 3.8]{Sta2023}) Any graph based on a well-ordered set of vertices has a chromatic index, a distinguishing number, and a distinguishing index.
				}
			\end{enumerate}
		}
	\end{fact}
	
	\subsection{Basic Results}
	
	\begin{prop}{($\mathsf{ZF}$)}
		{\em The Shelah-Soifer Graph $G = (\mathbb{R},\rho)$ has the following properties:
			\begin{enumerate}
				\item If $\mathsf{AC_{2}}$ holds, then $G$ has a minimal dominating set.
				
				\item Any independent set of $G$ is either non-measurable or of measure zero.
			\end{enumerate}
		}
	\end{prop}
	
	\begin{proof}
		First, we note that each component of $G$ is infinite, since $x,y\in \mathbb{R}$ are connected if and only if $x-y=q+\sqrt{2}z$ for some $q\in\mathbb{Q}$ and $z\in\mathbb{Z}$, and $G$ has no odd cycles. 
		
		(1). Under $\mathsf{AC_{2}}$, $G$ has a $2$-proper vertex coloring $f:V_{G}\rightarrow 2$ (see \cite{HR2005}). This is
		because, since $G$ has no odd cycles, each component of $G$ has precisely two 2-proper vertex colorings. Using $\mathsf{AC_{2}}$ one can select a 2-proper vertex coloring for each component, in order to obtain a 2-proper vertex coloring of $G$.
		We claim that $f^{-1}(i)$ (which is an independent set of $G$) is a maximal independent set (and hence a minimal dominating set) of $G$ for any $i\in \{0,1\}$. Fix $i\in \{0,1\}$ and assume that $f^{-1}(i)$ is not a maximal independent set. Then $f^{-1}(i)\cup \{v\}$ is an independent set for some $v\in \mathbb{R}\backslash f^{-1}(i)=f^{-1}(1-i)$ and so $\{v,x\}\not\in\rho$ for any $x\in f^{-1}(i)$. Since $f^{-1}(1-i)$ is an independent set, $\{v,x\}\not\in\rho$ for any $x\in f^{-1}(1-i)$. This contradicts the fact that $G$ has no isolated vertices. 
		
		(2). Let $M$ be an independent set of $G$. 
		Pick any $x,y\in M$ such that $x\neq y$. Then,
		
		\begin{center}
			$\neg(y\rho x)\implies (y-x)\not\in (\mathbb{Q}+\sqrt{2})\cup (\mathbb{Q}+(-\sqrt{2}))=\{r+\sqrt{2}:r\in\mathbb{Q}\}\cup\{r-\sqrt{2}:r\in\mathbb{Q}\}$.
		\end{center}
		
		Thus, there are no $x,y\in M$ where $x\neq y$ such that $y-x\in \mathbb{Q}+\sqrt{2}$. 
		By Fact 3.1(3), $M$ is not a measurable set of $\mathbb{R}$ with a positive measure. 
	\end{proof}

	\begin{prop}{($\mathsf{ZF}$)}
		{\em The following hold: 
			\begin{enumerate}
				\item Any graph based on a well-ordered set of vertices has a minimal vertex cover.
				
				\item Any graph based on a well-ordered set of vertices has a minimal dominating set.  
				
				\item Any graph based on a well-ordered set of vertices has a maximal matching.
				
				\item Any graph based on a well-ordered set of vertices with no isolated vertex, has a minimal edge cover.
			\end{enumerate}
		}
	\end{prop}
	
	\begin{proof}
		(1).  Let $G=(V_{G}, E_{G})$ be a graph based on a well-ordered set of vertices and let $<$ be a well-ordering of $V_{G}$.
		We use transfinite recursion, without invoking any form of choice, to construct a minimal vertex cover. Let $M_{0}=V_{G}$. Clearly, $M_{0}$ is a vertex cover. Assume that $M_{0}$ is not a minimal vertex cover.
		Now, assume that for some ordinal number $\alpha$ we
		have constructed a sequence $(M_{\beta})_{\beta<\alpha}$ of vertex covers such that $M_{\beta}$ is not a minimal vertex cover for any $\beta<\alpha$. If $\alpha=\gamma+1$ is a successor ordinal for some ordinal $\gamma$, then let $M_{\alpha}=M_{\gamma+1}=M_{\gamma}\backslash\{v_{\gamma}\}$   
		where $v_{\gamma}$ is the $<$-minimal element of the well-ordered set $\{v\in M_{\gamma}: M_{\gamma}\backslash \{v\}$ is a vertex cover$\}$. If $\alpha$ is a limit ordinal, we use $M_{\alpha}=\bigcap_{\beta\in \alpha} M_{\beta}$.
		For any ordinal $\alpha$, if $M_{\alpha}$ is a minimal vertex cover, then we are done. Since the class of all ordinal numbers is a proper class, it follows that the recursion must terminate at some ordinal stage, say $\lambda$. Then, $M_{\lambda}$ is the minimal vertex cover.
		
		
		(2). This follows from (1) and the fact that if $I$ is a minimal vertex cover of $G$, then $V_{G}\backslash I$ is a maximal
		independent set (and hence a minimal dominating set) of $G$.
		
		(3). If $V_{G}$ is well-orderable, then $E_{G}\subseteq [V_{G}]^{2}$ is well-orderable as well. Thus, similar to the arguments of (1) we can obtain a maximal matching by using transfinite recursion in $\mathsf{ZF}$ and modifying the greedy algorithm to construct a maximal matching.  
		
		(4). Let $G=(V_{G}, E_{G})$ be a graph on a well-ordered set of vertices without isolated vertices. Let $\prec'$
		be a well-ordering of $E_{G}$. By (3), we can obtain a maximal matching $M$ in $G$. Let $W$ be the set of vertices not covered by $M$. For each vertex $w\in W$, the set $E_{w}=\{e\in E_{G}: e$ is incident with $w\}$ is well-orderable being a subset of the well-orderable set $(E_{G},\prec')$.  Let $f_{w}$ be the $(\prec'\restriction E_{w})$-minimal element of $E_{w}$. Let $F=\{f_{w}:w\in W\}$ and let $M_{1}= \{e \in M:$ at least one endpoint of $e$ is not covered by $F\}$. Then $F\cup M_{1}$ is a minimal edge cover of $G$.
	\end{proof}
	
	\begin{remark}
		We remark that the direct
		proofs of items (1–3) of  Proposition 3.3 do not adapt immediately to give a proof of item (4);
		the issue is in the limit steps, where a vertex of infinite degree might not be covered
		anymore by the intersection of edge covers.
	\end{remark}
	
	\section{Proper and distinguishing colorings}
	\begin{defn}
		Let $\mathcal{A}=\{A_{n}:n\in\omega\}$ be a disjoint countably infinite family of non-empty finite sets and $T=\{t_{n}:n\in \omega\}$ be a countably infinite sequence disjoint from $A=\bigcup_{n\in\omega}A_{n}$. Let $G_{1}(\mathcal{A},T)=(V_{G_{1}(\mathcal{A},T)}, E_{G_{1}(\mathcal{A},T)})$ be the infinite locally finite connected graph such that
		
		\begin{flushleft}
			$V_{G_{1}(\mathcal{A},T)} := (\bigcup_{n\in \omega}A_{n})\cup T$, and
			
			$E_{G_{1}(\mathcal{A},T)} := \bigg\{\{t_{n}, t_{n+1}\}:{n\in\omega}\bigg\} 
			\cup 
			\bigg\{\{t_{n}, x\}: n\in\omega,x\in A_{n}\bigg\}
			\cup
			\bigg\{\{x,y\}:n\in \omega, x,y\in A_{n}, x\not=y\bigg\}$.
		\end{flushleft} 
	\end{defn}
	
	We denote by $\mathcal{C}$ the statement ``For any disjoint countably infinite family of non-empty finite sets $\mathcal{A}$, and any countably infinite sequence $T=\{t_{n}:n\in \omega\}$ disjoint from $A=\bigcup_{n\in\omega}A_{n}$, the graph $G_{1}(\mathcal{A},T)$ has a chromatic number" and we denote by 
	$\mathcal{C}_{k}$ the statement ``Any infinite locally finite connected graph $G$ such that $\delta(G)\geq k$ has a chromatic number".
	
	\begin{thm}{($\mathsf{ZF}$)}{\em Fix a natural number $k\geq 3$. The following statements are equivalent:
			\begin{enumerate}
				\item K\H{o}nig’s Lemma.
				\item $\mathcal{C}$. 
				\item $\mathcal{C}_{k}$.
				\item Any infinite locally finite connected graph has a chromatic number.
				\item Any infinite locally finite connected graph has a chromatic index.
				\item Any infinite locally finite connected graph has a distinguishing number.
				
				\item Any infinite locally finite connected graph has a distinguishing index.
			\end{enumerate}
		}
	\end{thm}
	
	\begin{proof}
		(1)$\Rightarrow$(2)-(7) 
		Let $G=(V_{G}, E_{G})$ be an infinite locally finite connected graph.
		Pick some $r \in V_{G}$. Let $V_{0}(r)=\{r\}$. For each integer $n \geq 1$, define $V_{n}(r) = \{v \in V_{G} : d_{G}(r, v) = n\}$ where ``$d_{G}(r, v) = n$'' means there are $n$ edges in the shortest path joining
		$r$ and $v$. Each $V_{n}(r)$ is finite by the local
		finiteness of $G$, and $V_{G} = \bigcup_{n\in \omega}V_{n}(r)$ since $G$ is connected. 
		By $\mathsf{AC^{\omega}_{fin}}$, $V_{G}$ is countably infinite (and hence, well-orderable). The rest follows from Fact 3.1(1,4) and the fact that  
		$G_{1}(\mathcal{A},T)$ is
		an infinite locally finite connected graph for any given disjoint countably infinite family $\mathcal{A}$ of non-empty finite sets and any countably infinite sequence $T=\{t_{n}:n\in \omega\}$ disjoint from $A=\bigcup_{n\in\omega}A_{n}$.
		
		(2)$\Rightarrow$(1) 
		Since $\mathsf{AC_{fin}^{\omega}}$ is equivalent to its partial version $\mathsf{PAC_{fin}^{\omega}}$ (Every countably infinite family of non-empty finite sets has an infinite
		subfamily with a choice function) (cf. \cite{HR1998}, \cite[the proof of Theorem 4.1(i)]{DHHKR2008} or footnote 4), it suffices to show that $\mathcal{C}$  implies $\mathsf{PAC_{fin}^{\omega}}$.
		In order to achieve this, we modify the arguments of Herrlich--Tachtsis \cite[Proposition 23]{HT2006} suitably.
		Let $\mathcal{A}=\{A_{n}:n\in \omega\}$ be a countably infinite set of non-empty finite sets without a partial choice function. Without loss of generality, we assume that $\mathcal{A}$ is disjoint. Pick a countably infinite sequence $T=\{t_{n}:n\in \omega\}$ disjoint from $A=\bigcup_{i\in\omega}A_{i}$ and consider the graph $G_{1}(\mathcal{A},T)=(V_{G_{1}(\mathcal{A},T)}, E_{G_{1}(\mathcal{A},T)})$ as in Figure 1.
		
		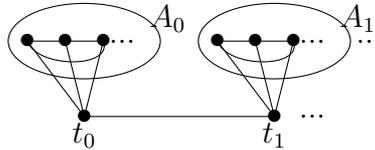
\begin{figure}[!ht]
			\centering
			\begin{minipage}{\textwidth}
				\centering
				\begin{tikzpicture}[scale=0.5]
					\draw (-2.5, 1) ellipse (2 and 1);
					\draw (-4,1) node {$\bullet$};
					\draw (-3,1) node {$\bullet$};
					\draw (-2,1) node {$\bullet$};
					\draw (-1.5,1) node {...};
					\draw (-0.3,1.6) node {$A_{0}$};
					\draw (-4,1) -- (-3,1);
					\draw (-3,1) -- (-2,1);
					\draw (-2,1) to[out=-70,in=-70] (-4,1);
					\draw (-2.5,-1) node {$\bullet$};
					\draw (-2.5,-1.5) node {$t_{0}$};
					\draw (-4,1) -- (-2.5,-1);
					\draw (-3,1) -- (-2.5,-1);
					\draw (-2,1) -- (-2.5,-1);
					\draw (2.5, 1) ellipse (2 and 1);
					\draw (1,1) node {$\bullet$};
					\draw (2,1) node {$\bullet$};
					\draw (3,1) node {$\bullet$};
					\draw (3.5,1) node {...};
					\draw (4.7,1.6) node {$A_{1}$};
					\draw (1,1) -- (2,1);
					\draw (2,1) -- (3,1);
					\draw (1,1) to[out=-70,in=-70] (3,1);
					\draw (2.5,-1) node {$\bullet$};
					\draw (2.5,-1.5) node {$t_{1}$};
					\draw (3,1) -- (2.5,-1);
					\draw (2,1) -- (2.5,-1);
					\draw (1,1) -- (2.5,-1);
					\draw (-2.5,-1) -- (2.5,-1);
					\draw (3.5,-1) node {...};
					\draw (5,1) node {...};
				\end{tikzpicture}
			\end{minipage}
			\caption{\em Graph $G_{1}(\mathcal{A},T)$, an infinite locally finite connected graph.}
		\end{figure}
		
		Let $f: V_{G_{1}(\mathcal{A},T)} \rightarrow C$ be a $C$-proper vertex coloring of $G_{1}(\mathcal{A},T)$, i.e., a map such that if $\{x,y\}\in E_{G_{1}(\mathcal{A},T)}$ then $f(x) \neq f(y)$. Then for each $c \in C$, the set $M_{c}=\{v\in f^{-1}(c): v\in A_{i}$ for some $i\in\omega\}$ must be finite, otherwise $M_{c}$ will generate a partial choice function for $\mathcal{A}$. 
		
		\begin{claim}
			{\em $f[\bigcup_{n\in \omega}A_{n}]$ is infinite.}
		\end{claim}
		
		\begin{proof}
			Otherwise, $\bigcup_{n\in \omega}A_{n}=\bigcup_{c\in f[\bigcup_{n\in \omega}A_{n}]} M_{c}$
			is finite since the finite union of finite sets is finite in $\mathsf{ZF}$ and we obtain a contradiction. 
		\end{proof}
		
		\begin{claim}
			{\em $f[\bigcup_{n\in \omega}A_{n}]$ is Dedekind-finite.}  
		\end{claim}
		
		\begin{proof}
			First, we note that $\bigcup_{n\in \omega}A_{n}$ is Dedekind-finite since $\mathcal{A}$ has no partial choice function. For the sake of contradiction, assume that $C=\{c_{i}:i\in \omega\}$ is a countably infinite subset of $f[\bigcup_{n\in \omega}A_{n}]$. 
			Fix a well-ordering $<$ of $\mathcal{A}$ (since $\mathcal{A}$ is countable, and hence well-orderable).
			Define $d_{i}$ to be the {\em unique} element of $M_{c_{i}} \cap A_{n}$ 
			where $n$ is the $<$-least element of $\{m\in\omega: M_{c_{i}} \cap A_{m}\neq \emptyset\}$.
			Such an $n$ exists since $c_{i} \in f[\bigcup_{n<\omega} A_{n}]$ and
			$M_{c_{i}} \cap A_{n}$ has a single element since $f$ is a proper vertex coloring.
			Then $\{d_{i}:i\in\omega\}$ is a countably infinite subset of $\bigcup_{n\in \omega}A_{n}$ which contradicts the fact that $\bigcup_{n\in \omega}A_{n}$ is Dedekind-finite.
		\end{proof}
		
		The following claim states that $\mathcal{C}$ fails.
		
		\begin{claim}
			{\em There is a $C_{1}$-proper vertex coloring $f:V_{G_{1}(\mathcal{A},T)}\rightarrow C_{1}$ of $G_{1}(\mathcal{A},T)$ such that $C_{1}\prec C$. Thus, $G_{1}(\mathcal{A},T)$ has no chromatic number.}  
		\end{claim}
		\begin{proof}
			Fix some $c_{0} \in f[\bigcup_{n\in \omega}A_{n}]$.
			Then $Index(M_{c_{0}}) = \{n \in \omega : M_{c_{0}} \cap A_{n} \neq \emptyset\}$ is finite. By claim 4.3, there exists some
			$b_{0} \in (f[\bigcup_{n\in \omega}A_{n}]\backslash \bigcup_{m\in Index(M_{c_{0}})} f[A_{m}])$ since the finite union of finite sets is finite. Define a proper vertex coloring $g: \bigcup_{n\in \omega}A_{n} \rightarrow (f[\bigcup_{n\in \omega}A_{n}]\backslash {c_{0}})$ as follows:
			
			\begin{center}
				$g(x) =
				\begin{cases} 
					f(x) & \text{if}\, f(x)\neq c_{0}, \\
					
					b_{0} & \text{otherwise}.
				\end{cases}$
			\end{center}
			
			Similarly, we can define a proper vertex coloring $h: \bigcup_{n\in \omega}A_{n}\rightarrow (f[\bigcup_{n\in \omega}A_{n}]\backslash \{ c_{0},c_{1},c_{2}\})$ for some $c_{0}, c_{1},c_{2}\in f[\bigcup_{n\in \omega}A_{n}]$. Let $h(t_{2n})=c_{0}$ and $h(t_{2n+1})=c_{1}$ for all $n\in\omega$. 
			Thus, $h:V_{G_{1}} \rightarrow (f[\bigcup_{n\in \omega}A_{n}]\backslash \{c_{2}\})$ is a $f[\bigcup_{n\in \omega}A_{n}]\backslash \{c_{2}\}$-proper vertex coloring of $G_{1}$. 
			We define $C_{1}=f[\bigcup_{n\in\omega} A_{n}]\backslash \{c_{2}\}$. By claim 4.4,
			$C_{1} \prec f[\bigcup_{n\in\omega} A_{n}] \preceq C$. 
		\end{proof}
		
		Similarly, we can see (4)$\Rightarrow$(1).
		
		(3)$\Rightarrow$(1)  
		Let $\mathcal{A}=\{A_{n}:n\in \omega\}$ be a disjoint countably infinite set of non-empty finite sets without a partial choice function, such that $k$ divides $\vert A_{n}\vert$ for each $n\in\omega$ and $k\in\omega\backslash\{0,1\}$. Assume $T$ and $G_{1}(\mathcal{A},T)$ as in the proof of (2)$\Rightarrow$(1).
		Then $\delta(G_{1}(\mathcal{A},T))\geq k$.
		By the arguments of (2)$\Rightarrow$(1), $\mathcal{C}$ implies $\mathsf{PAC^{\omega}_{\text{$k$}\times fin}}$. 
		Following the arguments of \cite[Theorem 4.1]{DHHKR2008}, we can see that $\mathsf{PAC^{\omega}_{\text{$k$}\times fin}}$ implies $\mathsf{AC^{\omega}_{fin}}$.\footnote{For the reader’s convenience, we write down the proof. First, we can see that $\mathsf{PAC^{\omega}_{\text{$k$}\times fin}}$ implies $\mathsf{AC^{\omega}_{\text{$k$}\times fin}}$.
			Fix a family $\mathcal{A} = \{A_{i}:i\in \omega\}$ of disjoint nonempty finite sets such that $k$ divides $\vert A_{i} \vert$  for each $i\in\omega$. Then the family
			\begin{center}
				$\mathcal{B} = \{B_{i} : i \in \omega\}$ where $B_{i} = \prod_{j\leq i} A_{j}$   
			\end{center}
			is a disjoint family such that $k$ divides $\vert B_{i} \vert$ and any partial choice function on $\mathcal{B}$ yields a choice function for $\mathcal{A}$.
			
			Finally, fix a family $\mathcal{C} = \{C_{i} : i \in \omega\}$
			of disjoint nonempty finite sets. Then
			$\mathcal{D} = \{D_{i}: i \in \omega\}$ where $D_{i}= C_{i} \times k$
			is a pairwise disjoint family of finite sets where $k$ divides $\vert D_{i} \vert$ for each $i\in\omega$. Thus $\mathsf{AC^{\omega}_{\text{$k$}\times fin}}$ implies that $\mathcal{D}$ has a choice function $f$ which determines a choice function for $\mathcal{C}$.}
		
		(5)$\Rightarrow$(1) 
		Let $\mathcal{A}=\{A_{n}:n\in \omega\}$ be a disjoint countably infinite set of non-empty finite sets without a partial choice function and $T=\{t_{n}:n\in \omega\}$ be a sequence 
		disjoint from $A=\bigcup_{n\in\omega}A_{n}$. 
		Let $H_1$ be the graph obtained from the graph $G_{1}(\mathcal{A},T)$ of (2)$\Rightarrow$(1) after deleting the edge set $\{\{x,y\}: n\in \omega, x,y\in A_n, x\neq y\}$. Clearly, $H_1$ is an infinite locally finite connected graph.
		
		\begin{claim}
			{\em $H_1$ has no chromatic index.}
		\end{claim}
		
		\begin{proof}
			Assume that the graph $H_{1}$ has a chromatic index. Let $f: E_{H_1}\to C$ be a proper edge coloring with $|C|=\kappa$, where $\kappa$ is the chromatic index of $H_1$. Let 
			$B = \{\{t_n,x\}: n\in \omega,  x\in A_n\}$.  
			Similar to claims 4.3, 4.4, and 4.5, $f[B]$ is an infinite, Dedekind-finite set and there is a proper 
			edge coloring $h: B \to f[B]\setminus\{c_0, c_1, c_2\}$  for some $c_0, c_1,c_2\in f[B]$. Finally, define $h(\{t_{2n}, t_{2n+1}\})= c_0$ and $h(\{t_{2n+1}, t_{2n+2}\})= c_1$ for all $n\in \omega$. Thus, we obtain a $f[B]\setminus\{c_2\}$-proper edge coloring $h:E_{H_1}\to f[B]\setminus\{c_2\}$, with 
			$f[B]\setminus\{c_2\} \prec f[B]\preceq C$ 
			as $f[B]$ is Dedekind-finite, contradicting the fact that $\kappa$ is the chromatic index of $H_1$.
		\end{proof}
		
		(6)$\Rightarrow$(1) Assume $\mathcal{A}$ and $T$ as in the proof of (5)$\Rightarrow$(1). Let $H_1^{1}$ be the graph obtained from $H_1$ of (5)$\Rightarrow$(1) by adding two new vertices $t'$ and $t''$ and the edges $\{t'', t'\}$ and $\{t', t_0\}$. 
		
		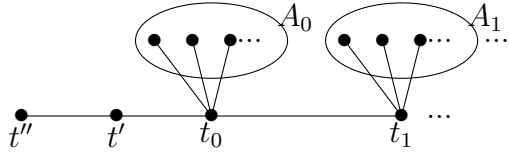
\begin{figure}[!ht]
			\centering
			\begin{minipage}{\textwidth}
				\centering
				\begin{tikzpicture}[scale=0.5]
					\draw (-2.5, 1) ellipse (2 and 1);
					\draw (-4,1) node {$\bullet$};
					\draw (-3,1) node {$\bullet$};
					\draw (-2,1) node {$\bullet$};
					\draw (-1.5,1) node {...};
					\draw (-0.3,1.6) node {$A_{0}$};
					\draw (-2.5,-1) node {$\bullet$};
					\draw (-2.5,-1.5) node {$t_{0}$};
					\draw (-4,1) -- (-2.5,-1);
					\draw (-3,1) -- (-2.5,-1);
					\draw (-2,1) -- (-2.5,-1);
					\draw (2.5, 1) ellipse (2 and 1);
					\draw (1,1) node {$\bullet$};
					\draw (2,1) node {$\bullet$};
					\draw (3,1) node {$\bullet$};
					\draw (3.5,1) node {...};
					\draw (4.7,1.6) node {$A_{1}$};
					\draw (2.5,-1) node {$\bullet$};
					\draw (2.5,-1.5) node {$t_{1}$};
					\draw (3,1) -- (2.5,-1);
					\draw (2,1) -- (2.5,-1);
					\draw (1,1) -- (2.5,-1);
					\draw (-2.5,-1) -- (2.5,-1);
					\draw (3.5,-1) node {...};
					\draw (5,1) node {...};
					
					\draw (-5,-1) node {$\bullet$};
					\draw (-5,-1.5) node {$t'$};
					\draw (-5,-1) -- (-2.5,-1);
					\draw (-7.5,-1) node {$\bullet$};
					\draw (-7.5,-1.5) node {$t''$};
					\draw (-7.5,-1) -- (-5,-1);
				\end{tikzpicture}
			\end{minipage}
			\caption{\em Graph $H^{1}_{1}$, an infinite locally finite connected graph.}
		\end{figure}
		It suffices to show that $H_1^1$ has no distinguishing number.
		We recall the fact that whenever $j: V_{H_1^1}\to V_{H_1^1}$ is an automorphism, $\varphi(x_{1},..., x_{r})$ is a first-order $\mathcal{L}$-formula on $r$ variables (where $\mathcal{L}$ is the language of graphs) for some $r\in\omega\backslash\{0\}$, and $a_{i}\in V_{H_1^1}$ for each $1\leq i\leq r$, then $H_1^1\models\varphi(a_{1},..., a_{r})$ if and only if $H_1^1\models\varphi(j(a_{1}),..., j(a_{r}))$ (cf. Definition 2.8). 
		
		\begin{claim}
			{\em $t', t''$, and $t_m$ are fixed by every automorphism for each non-negative integer $m$.}
		\end{claim}
		\begin{proof}
			Fix non-negative integers $n,m,r$. The first-order $\mathcal{L}$-formula
			\[ \mathsf{Deg}_n(x) := \exists x_0\dots \exists x_{n-1} \big(\bigwedge_{i\neq j}^{n-1}x_i\neq x_j\wedge  \bigwedge_{i<n}x\neq x_i \wedge \bigwedge_{i<n} Exx_i \wedge \forall y(Exy \to \bigvee_{i<n}y=x_i)\big)\]
			expresses the property that a vertex $x$ has degree $n$, where $Eab$ denotes the existence of an edge between vertices $a$ and $b$.  We define the following first-order $\mathcal{L}$-formula:
			
			\begin{flushleft}
				$\varphi(x) := \mathsf{Deg}_1(x)\wedge \exists y(Exy \wedge \mathsf{Deg}_2(y))$.
			\end{flushleft}
			
			It is easy to see the following:
			
			\begin{enumerate}
				\item[(i)] $t''$ is the unique vertex such that $H_1^1\models \varphi(t'')$. This means $t''$ is the unique vertex such that $\deg(t'')=1$ and $t''$ has a neighbor of degree $2$.
				\item[(ii)] $t'$ is the unique vertex such that $H_1^1\models  \mathsf{Deg}_2(t')$. So $t'$ is the unique vertex with $\deg(t')= 2$.
			\end{enumerate}
			Fix any automorphism $\tau$. Since every automorphism preserves the properties mentioned in (i)--(ii), $t'$ and $t''$ are fixed by $\tau$. 
			The vertices $t_{m}$ are fixed by $\tau$ by induction as follows: Since $t_{i}$ is the unique vertex of path length $i+1$ from $t''$ such that the degree of $t_{i}$ is greater than 1, where $i\in \{0,1\}$, we have that $t_{0}$ and $t_{1}$ are fixed by $\tau$.
			Assume that $\tau(t_{l})=t_{l}$ for all $l<m-1$. We show that $\tau(t_{m})=t_{m}$. Now, $\tau(t_{m})$ is a neighbour of $\tau(t_{m-1}) = t_{m-1}$ which is of degree at least $2$, so $\tau(t_{m})$ must be either $t_{m-2}$ or $t_{m}$, but $t_{m-2} = \tau(t_{m-2})$ is already taken. So, $\tau(t_{m})=t_{m}$.
		\end{proof}
		
		\begin{claim}
			{\em Fix $m\in \omega$ and $x\in A_m$. Then $Orb_{Aut(H_1^1)}(x)=\{g(x): g\in Aut(H_1^1)\}=A_m$.} 
		\end{claim}
		
		\begin{proof}
			This follows from the fact that each $y\in \bigcup_{n\in\omega}A_n$ has path length $1$ from $t_{m}$ if and only if $y\in A_m$.       
		\end{proof}
		
		\begin{claim}
			{\em $H_1^1$ has no distinguishing number.}
		\end{claim}
		\begin{proof}
			Assume that the graph $H_{1}^{1}$ has a distinguishing number. Let  $f: V_{H_1^1} \to C$ be a distinguishing vertex coloring with $|C|=\kappa$, where $\kappa$ is the distinguishing number of $H_1^1$. Similar to claims 4.3 and 4.4, $f[\bigcup_{n\in \omega}A_n]$ is infinite and Dedekind-finite. Consider a coloring $h: \bigcup_{n\in \omega}A_n\to f[\bigcup_{n\in \omega}A_n]\setminus\{c_0,c_{1},c_{2}\}$ for some $c_{0}, c_{1}, c_{2}\in f[\bigcup_{n\in \omega}A_{n}]$, just as in claim 4.5. Let $h(t)=c_{0}$ for all $t\in \{t'',t'\}\cup T$. Then, $h: V_{H^1_{1}} \rightarrow (f[\bigcup_{n\in \omega}A_{n}]\backslash \{c_{1},c_{2}\})$ is a $f[\bigcup_{n\in \omega}A_n]\setminus\{c_1,c_{2}\}$-distinguishing vertex coloring of $H^1_{1}$. 
			Finally, 
			$f[\bigcup_{n\in \omega}A_n]\setminus\{c_1,c_{2}\} \prec  f[\bigcup_{n\in\omega} A_{n}] \preceq C$
			contradicts the fact that $\kappa$ is the distinguishing number of $H_1^1$.
		\end{proof}
		
		(7)$\Rightarrow$(1) Assume $\mathcal{A}$, $T$, and $H^1_1$ as in the proof of (6)$\Rightarrow$(1). By claim 4.7, every automorphism fixes the edges $\{t'',t'\}$, $\{t', t_0\}$ and $\{t_n, t_{n+1}\}$ for each $n\in \omega$. 
		Moreover, if $H_1^1$ has a distinguishing edge coloring $f$, then for each $n\in \omega$ and $x,y\in A_n$ such that $x\neq y$, $f(\{t_n, x\})\neq f(\{t_n, y\})$.
		
		\begin{claim}
			{\em $H_1^1$ has no distinguishing index.}
		\end{claim}
		\begin{proof}
			This follows modifying the arguments of claims 4.6 and 4.9.     
		\end{proof}
	\end{proof}
	
	\section{Irreducible proper coloring and covering properties}
	\begin{thm}{($\mathsf{ZF}$)}{\em The following statements are equivalent:
			\begin{enumerate}
				\item K\H{o}nig’s Lemma.
				\item Every infinite locally finite connected graph has an irreducible proper coloring.
				
				\item Every infinite locally finite connected graph has a minimal dominating set. 
				
				\item Every infinite locally finite connected graph has a minimal edge cover.
				
				\item Every infinite locally finite connected graph has a maximal matching.
			\end{enumerate}
		}
	\end{thm}
	
	\begin{proof}
		Implications (1)$\Rightarrow$(2)-(5) follow from Proposition 3.3, and the fact that $\mathsf{AC_{fin}^{\omega}}$ implies every infinite locally finite connected graph is countably infinite.
		
		(2)$\Rightarrow$(1) In view of the proof of Theorem 4.2 ((2)$\Rightarrow$(1)), it suffices to show that the given statement implies $\mathsf{PAC_{fin}^{\omega}}$.
		Let $\mathcal{A}=\{A_{n}:n\in \omega\backslash\{0\}\}$ be a disjoint countably infinite set of non-empty finite sets without a partial choice function. Pick $t\not\in \bigcup_{i\in\omega\backslash\{0\}}A_{i}$. Let $A_{0}=\{t\}$. Consider the following  infinite locally finite connected graph $G_{2}=(V_{G_{2}}, E_{G_{2}})$ (see Figure 3):
		
		\begin{figure}[!ht]
			\centering
			\begin{minipage}{\textwidth}
				\centering
				\begin{tikzpicture}[scale=0.5]
					\draw (-6, 1) node {$\bullet$};
					\draw (-6, 1.5) node {$t$};
					
					\draw (-3,1) ellipse (1 and 2);
					\draw (-3,2) node {$\bullet$};
					\draw (-3,1) node {$\bullet$};
					\draw (-3,0) node {$\bullet$};
					\draw (-1.7,2.5) node {$A_{1}$};
					\draw (-3, 1) -- (-3,2);
					\draw (-3, 0) -- (-3,1);
					\draw (-3,0) to[in=20, out=20] (-3,2);
					\draw (0, 1) -- (0,2);
					\draw (3, 1) -- (3,2);
					\draw (3, 0) -- (3,1);
					\draw (3,0) to[in=20, out=20] (3,2);
					\draw (6, 1) -- (6,2);
					\draw (6, 0) -- (6,1);
					\draw (6,0) to[in=20, out=20] (6,2);
					\draw (-6, 1) -- (-3,2);
					\draw (-6, 1) -- (-3,1);
					\draw (-6, 1) -- (-3,0);
					
					\draw (0, 1) ellipse (1 and 2);
					\draw (0,2) node {$\bullet$};
					\draw (0,1) node {$\bullet$};
					\draw (1.3,2.5) node {$A_{2}$};
					
					\draw (-3, 0) -- (0,2);
					\draw (-3, 0) -- (0,1);
					\draw (-3, 1) -- (0,2);
					\draw (-3, 1) -- (0,1);
					\draw (-3, 2) -- (0,2);
					\draw (-3, 2) -- (0,1);
					
					\draw (3, 1) ellipse (1 and 2);
					\draw (3,2) node {$\bullet$};
					\draw (3,1) node {$\bullet$};
					\draw (3,0) node {$\bullet$};
					\draw (4.3,2.5) node {$A_{3}$};
					
					\draw (0, 1) -- (3,2);
					\draw (0, 1) -- (3,1);
					\draw (0, 1) -- (3,0);
					\draw (0, 2) -- (3,2);
					\draw (0, 2) -- (3,1);
					\draw (0, 2) -- (3,0);
					
					\draw (6, 1) ellipse (1 and 2);
					\draw (6,2) node {$\bullet$};
					\draw (6,1) node {$\bullet$};
					\draw (6,0) node {$\bullet$};
					\draw (7.3,2.5) node {$A_{4}$};
					
					\draw (3, 0) -- (6,2);
					\draw (3, 0) -- (6,1);
					\draw (3, 0) -- (6,0);
					\draw (3, 1) -- (6,2);
					\draw (3, 1) -- (6,1);
					\draw (3, 1) -- (6,0);
					\draw (3, 2) -- (6,2);
					\draw (3, 2) -- (6,1);
					\draw (3, 2) -- (6,0);
					
					\draw (8,1) node {...};
				\end{tikzpicture}
			\end{minipage}
			\caption{\em The graph $G_{2}$ when $\vert A_{1}\vert = \vert A_{3}\vert = \vert A_{4}\vert=3$, and $\vert A_{2}\vert=2$.}
		\end{figure}
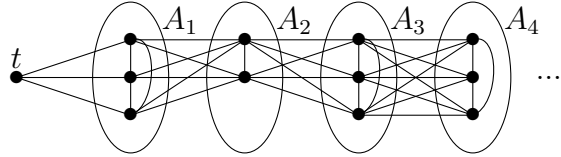
		
		\begin{flushleft}
			$V_{G_{2}} := \bigcup_{n\in \omega}A_{n}$, 
			
			$E_{G_{2}} := \bigg\{\{x, y\}: n\in\omega\backslash\{0\}, x,y\in A_{n}, x\not=y\bigg\}$
			$\cup \bigg\{\{x,y\}:n\in \omega\backslash\{0\}, x\in A_{n}, y\in A_{n+1}\bigg\}$
			$\cup \bigg\{\{t,x\}:x\in A_{1}\bigg\}$.
		\end{flushleft}
		
		\begin{claim}
			{\em $G_{2}$ has no irreducible proper coloring.}
		\end{claim}
		\begin{proof}
			Let $f: V_{G_{2}} \rightarrow C$ be a $C$-irreducible proper coloring of $G_{2}$, i.e., a map such that $f(x) \neq f(y)$ if $\{x,y\}\in E_{G_{2}}$ and $(\forall c_{1},c_{2}\in C) f^{-1}(c_{1})\cup f^{-1}(c_{2})$ is dependent. Similar to the proof of Theorem 4.2((2)$\Rightarrow$(1)), $f^{-1}(c)$ is finite for all $c \in C$, and $f[\bigcup_{n\in \omega\backslash\{0\}}A_{n}]$ is infinite.
			Fix $c_{0} \in f[\bigcup_{n\in \omega\backslash\{0\}}A_{n}]$. Then $Index(f^{-1}(c_{0})) = \{n \in \omega\backslash \{0\} : f^{-1}(c_{0}) \cap A_{n} \neq \emptyset\}$ is finite. So there exists some
			
			\begin{center}
				$c_{1} \in f[\bigcup_{n\in \omega\backslash \{0\}}A_{n}]\backslash \bigcup_{m\in Index(f^{-1}(c_{0}))} (f[A_{m}]\cup f[A_{m-1}] \cup f[A_{m+1}])$
			\end{center}
			
			as $\bigcup_{m\in Index(f^{-1}(c_{0}))} (f[A_{m}]\cup f[A_{m-1}] \cup f[A_{m+1}])$ is finite. Clearly, $f^{-1}(c_{0})\cup f^{-1}(c_{1})$ is independent, and we obtain a contradiction. 
		\end{proof}
		
		(3)$\Rightarrow$(1) Assume $\mathcal{A}$ as in the proof of (2)$\Rightarrow$(1). 
		Let $G_2^{1}$ be the infinite locally finite connected graph obtained from $G_2$ of (2)$\Rightarrow$(1) after deleting $t$ and $\{\{t,x\}:x\in A_{1}\}$.
		Consider a minimal dominating set $D$ of $G_2^{1}$. The following conditions must be satisfied:
		\begin{enumerate}
			\item[(i)] Since $D$ is a dominating set, for each $n\in \omega\setminus\{0,1\}$, there is an $a\in D$ such that $a\in A_{n-1}\cup A_n \cup A_{n+1}$ (otherwise, no vertices from $A_n$ belongs to $D$ or have a neighbor in $D$).
			\item[(ii)] By the minimality of $D$, we have $|A_n \cap D| \leq 1$ for each $n\in \omega\setminus\{0\}$.
		\end{enumerate}
		Clearly, (i) and (ii) determine a partial choice function over $\mathcal{A}$, contradicting the assumption that $\mathcal{A}$ has no partial choice function.
		
		(4)$\Rightarrow$(1) Let $\mathcal{A}=\{A_{n}: n\in \omega\}$ be a disjoint countably infinite set of non-empty finite sets and let $A=\bigcup_{n\in \omega} A_{n}$. Consider a countably infinite family $(B_{i},<_{i})_{i\in \omega}$ of well-ordered sets such that the following hold (cf. the proof of \cite[Theorem 1, Remark 6]{DM2006}):
		
		\begin{enumerate}[(i)]
			\item  $\vert B_{i}\vert=\vert A_{i}\vert + k$ for some fixed $1\leq k\in \omega$ and thus, there is no mapping with domain $A_{i}$ and range $B_{i}$.
			\item for each $i\in \omega$, $B_{i}$ is disjoint from $A$ and the other $B_{j}$’s.
		\end{enumerate}
		
		Let $B=\bigcup_{i\in \omega} B_{i}$. Pick a countably infinite sequence $T=\{t_{i}:i\in \omega\}$ disjoint from $A$ and $B$ and consider the following infinite locally finite connected graph $G_{3}=(V_{G_{3}}, E_{G_{3}})$:
		
		\begin{figure}[!ht]
			\centering
			\begin{minipage}{\textwidth}
				\centering
				\begin{tikzpicture}[scale=0.6]
					\draw (-2.5, 1) ellipse (2 and 1);
					\draw (-4,1) node {$\bullet$};
					\draw (-2,1) node {$\bullet$};
					\draw (-1.5,1) node {...};
					\draw (-0.3,1.6) node {$A_{0}$};
					\draw (-2.5,3) node {$\bullet$};
					\draw (-2.5,3.5) node {$t_{0}$};
					\draw (-2.5, -1) ellipse (2 and 0.5);
					\draw (-0.5,-0.3) node {$(B_{0},<_{0})$};
					\draw (-4,-1) node {$\bullet$};
					\draw (-2,-1) node {$\bullet$};
					\draw (-1.5,-1) node {...};
					\draw (-4,1) -- (-2.5,3);
					\draw (-2,1) -- (-2.5,3);
					\draw (-4,1)-- (-4,-1);
					\draw (-2,1)-- (-2,-1);
					\draw (-4,1)-- (-2,-1);
					\draw (-2,1)-- (-4,-1);
					\draw (2.5, 1) ellipse (2 and 1);
					\draw (1,1) node {$\bullet$};
					\draw (3,1) node {$\bullet$};
					\draw (3.5,1) node {...};
					\draw (4.7,1.6) node {$A_{1}$};
					\draw (2.5,3) node {$\bullet$};
					\draw (2.5,3.5) node {$t_{1}$};
					\draw (2.5, -1) ellipse (2 and 0.5);
					\draw (4.5,-0.3) node {$(B_{1},<_{1})$};
					\draw (1,-1) node {$\bullet$};
					\draw (3,-1) node {$\bullet$};
					\draw (3.5,-1) node {...};
					\draw (3,1) -- (2.5,3);
					\draw (1,1) -- (2.5,3);
					\draw (3,1)-- (3,-1);
					\draw (1,1)-- (1,-1);
					\draw (3,1)-- (1,-1);
					\draw (1,1)-- (3,-1);
					\draw (7.5, 1) ellipse (2 and 1);
					\draw (6,1) node {$\bullet$};
					\draw (8,1) node {$\bullet$};
					\draw (8.5,1) node {...};
					\draw (9.8,1.5) node {$A_{2}$};
					\draw (7.5, -1) ellipse (2 and 0.5);
					\draw (9.5,-0.3) node {$(B_{2},<_{2})$};
					\draw (6,-1) node {$\bullet$};
					\draw (8,-1) node {$\bullet$};
					\draw (8.5,-1) node {...};
					\draw (8,1) -- (7.5,3);
					\draw (6,1) -- (7.5,3);
					\draw (7.5,3) node {$\bullet$};
					\draw (7.5,3.5) node {$t_{2}$};
					\draw (8,1)-- (8,-1);
					\draw (6,1)-- (6,-1);
					\draw (8,1)-- (6,-1);
					\draw (6,1)-- (8,-1);
					\draw (-2.5,3) -- (2.5,3);
					\draw (2.5,3) -- (7.5,3);
					\draw (10,1) node {...};
					\draw (10,-1) node {...};
					\draw (8,3) node {...};
				\end{tikzpicture}
			\end{minipage}
			\caption{\em Graph $G_{3}$}
		\end{figure}
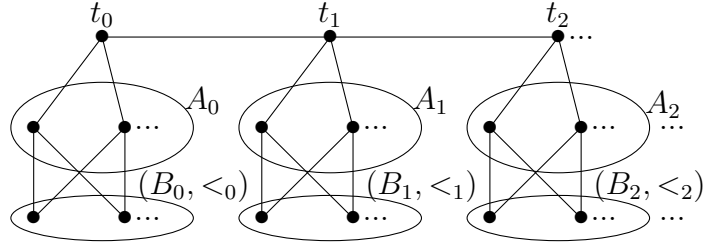
		
		\begin{flushleft}
			$V_{G_{3}} := A\cup B\cup T$, 
			
			$E_{G_{3}} := \bigg\{\{t_{i}, t_{i+1}\}:{i\in\omega}\bigg\} 
			\text{ }\cup  \text{ }
			\bigg\{\{t_{i}, x\}: i\in\omega,x\in A_{i}\bigg\}
			\text{ }\cup \text{ }
			\bigg\{\{x,y\}:i\in \omega, x\in A_{i},y\in B_{i}\bigg\}$.
		\end{flushleft}
		
		By assumption, $G_{3}$ has a minimal edge cover, say $G'_{3}$.
		For each $i \in \omega$, let $f_{i}: B_{i} \rightarrow \mathcal{P}(A_{i})\backslash\{\emptyset\}$ map each vertex of $B_{i}$ to its neighborhood in $G'_{3}$. 
		
		\begin{claim}
			{\em Fix $i\in\omega$. For any two distinct $\epsilon_{1}$ and $\epsilon_{2}$ in $B_{i}$, $\vert f_{i}(\epsilon_{1})\cap f_{i}(\epsilon_{2})\vert \leq 1$. } 
		\end{claim}
		\begin{proof}
			This follows from the fact that $G'_{3}$ does not contain 
			a complete bipartite subgraph $K_{2,2}$. In particular, each component of $G'_{3}$ has at most one vertex of degree greater than $1$.
			If any edge $e\in G'_{3}$ has both of its endpoints incident on edges of $G'_{3} \backslash \{e\}$, then $G'_{3} \backslash \{e\}$ is also an edge
			cover of $G_{3}$, contradicting the minimality of $G'_{3}$.
		\end{proof}
		By Fact 3.1(2) and (i), there are tuples $(\epsilon'_{1}, \epsilon'_{2})\in B_{i}\times B_{i}$ such that $f_{i}(\epsilon'_{1})\cap f_{i}(\epsilon'_{2})\neq \emptyset$. Consider the first such tuple $(\epsilon''_{1},\epsilon''_{2})$ with respect to the lexicographical ordering of $B_{i}\times B_{i}$. Then $\{f_{i}(\epsilon''_{1})\cap f_{i}(\epsilon''_{2}): i\in \omega\}$ is a choice function of $\mathcal{A}$ by claim 5.3.
		
		(5)$\Rightarrow$(1) Assume $\mathcal{A}$, and $A$ as in the proof of (4)$\Rightarrow$(1). Let $R=\{r_{n}:n\in \omega\}$ and $T=\{t_{n}:n\in \omega\}$ be two disjoint countably infinite sequences disjoint from $A$.
		We define the following locally finite connected graph $G_{4}=(V_{G_{4}},E_{G_{4}})$ (see Figure 5):
		
		\begin{figure}[!ht]
			\centering
			\begin{minipage}{\textwidth}
				\centering
				\begin{tikzpicture}[scale=0.5]
					\draw (-2.5, 1) ellipse (2 and 1);
					\draw (-4,1) node {$\bullet$};
					\draw (-3,1) node {$\bullet$};
					\draw (-2,1) node {$\bullet$};
					\draw (-1.5,1) node {...};
					\draw (-2.5,2.6) node {$\bullet$};
					\draw (-2.5,3.1) node {$r_{0}$};
					\draw (-4,1) -- (-2.5,2.6);
					\draw (-3,1) -- (-2.5,2.6);
					\draw (-2,1) -- (-2.5,2.6);
					\draw (-0.2,1.6) node {$A_{0}$};
					\draw (-2.5,-1) node {$\bullet$};
					\draw (-2.5,-1.5) node {$t_{0}$};
					\draw (-4,1) -- (-2.5,-1);
					\draw (-3,1) -- (-2.5,-1);
					\draw (-2,1) -- (-2.5,-1);
					\draw (2.5, 1) ellipse (2 and 1);
					\draw (1,1) node {$\bullet$};
					\draw (2,1) node {$\bullet$};
					\draw (3,1) node {$\bullet$};
					\draw (3.5,1) node {...};
					\draw (2.5,2.6) node {$\bullet$};
					\draw (2.5,3.2) node {$r_{1}$};
					\draw (3,1) -- (2.5,2.6);
					\draw (2,1) -- (2.5,2.6);
					\draw (1,1) -- (2.5,2.6);
					\draw (4.8,1.6) node {$A_{1}$};
					\draw (2.5,-1) node {$\bullet$};
					\draw (2.5,-1.5) node {$t_{1}$};
					\draw (3,1) -- (2.5,-1);
					\draw (2,1) -- (2.5,-1);
					\draw (1,1) -- (2.5,-1);
					\draw (-2.5,-1) -- (2.5,-1);
					\draw (3.5,-1) node {...};
					\draw (5,1) node {...};
				\end{tikzpicture}
			\end{minipage}
			\caption{\em Graph $G_{4}$.}
		\end{figure}
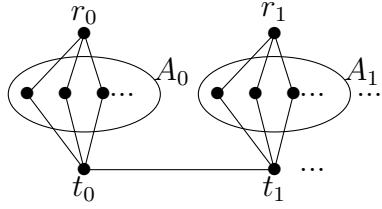
		\begin{flushleft}
			$V_{G_{4}} := (\bigcup_{n\in \omega}A_{n})\cup R\cup T$,
			
			$E_{G_{4}} := \bigg\{\{t_{n}, t_{n+1}\}:{n\in\omega}\bigg\} 
			\text{ }\cup  \text{ }
			\bigg\{\{t_{n}, x\}: n\in\omega,x\in A_{n}\bigg\}
			\text{ }\cup \text{ }
			\bigg\{\{r_{n},x\}:n\in \omega, x\in A_{n}\bigg\}$. 
		\end{flushleft}
		
		Let $M$ be a maximal matching of $G_{4}$.   
		For all $i\in \omega$, there is at most one $x\in A_{i}$ such that $\{r_{i},x\}\in M$ since $M$ is a matching and there is at least one $x\in A_{i}$ such that $\{r_{i},x\}\in M$ since $M$ is maximal. These unique $x\in A_{i}$ determine a choice function for $\mathcal{A}$. 
		
		This concludes the proof of the Theorem.
	\end{proof}
	
	\section{Remarks on new equivalents of AC}
	\begin{remark} We remark that the statement ``Any connected graph has a minimal dominating set" implies $\mathsf{AC}$.\footnote{The authors are very thankful to one of the referees for pointing out to us an error that appeared in this remark in a former version of the paper and especially for guiding us to eliminate the error.}
		Consider a family $\mathcal{A}=\{A_{i}:i\in I\}$ of pairwise disjoint non-empty sets. For each $i\in I$, let $B_i^0= A_{i}\times \{0\}$ and $B_i^1= A_{i}\times \{1\}$. Pick $t \not\in \bigcup_{i\in I} B_i^0\cup \bigcup_{i\in I} B_i^1$ and consider the following connected graph $G_{5}=(V_{G_{5}}, E_{G_{5}})$:
		
		\begin{figure}[!ht]
			\centering
			\begin{minipage}{\textwidth}
				\centering
				\begin{tikzpicture}[scale=0.5]
					\draw (-2.5, 4) ellipse (2 and 1);
					\draw (-4,4) node {$\bullet$};
					\draw (-3,4) node {$\bullet$};
					\draw (-2,4) node {$\bullet$};
					\draw (-1.5,4) node {...};
					\draw (-0.1,5.3) node {$B_0^1$};
					\draw (-4,4) -- (-3,4);
					\draw (-3,4) -- (-2,4);
					\draw (-2,4) to[out=-70,in=-70] (-4,4);
					
					\draw (2.5, 4) ellipse (2 and 1);
					\draw (1,4) node {$\bullet$};
					\draw (2,4) node {$\bullet$};
					\draw (3,4) node {$\bullet$};
					\draw (3.5,4) node {...};
					\draw (4.9,5.3) node {$B_1^1$};
					\draw (1,4) -- (2,4);
					\draw (2,4) -- (3,4);
					\draw (3,4) to[out=-70,in=-70] (1,4);
					
					\draw (7.5, 4) ellipse (2 and 1);
					\draw (6,4) node {$\bullet$};
					\draw (7,4) node {$\bullet$};
					\draw (8,4) node {$\bullet$};
					\draw (8.5,4) node {...};
					\draw (9.9,5.3) node {$B_2^1$};
					\draw (6,4) -- (7,4);
					\draw (7,4) -- (8,4);
					\draw (8,4) to[out=-70,in=-70] (6,4);
					\draw (-4,4) -- (-3,1);
					\draw (-3,4) -- (-3,1);
					\draw (-2,4) -- (-3,1);
					\draw (-4,4) -- (-4,1);
					\draw (-3,4) -- (-4,1);
					\draw (-2,4) -- (-4,1);
					\draw (-4,4) -- (-2,1);
					\draw (-3,4) -- (-2,1);
					\draw (-2,4) -- (-2,1);
					\draw (1,4) -- (1,1);
					\draw (2,4) -- (1,1);
					\draw (3,4) -- (1,1);
					\draw (1,4) -- (2,1);
					\draw (2,4) -- (2,1);
					\draw (3,4) -- (2,1);
					\draw (1,4) -- (3,1);
					\draw (2,4) -- (3,1);
					\draw (3,4) -- (3,1);
					\draw (6,4) -- (6,1);
					\draw (7,4) -- (6,1);
					\draw (8,4) -- (6,1);
					\draw (6,4) -- (7,1);
					\draw (7,4) -- (7,1);
					\draw (8,4) -- (7,1);
					\draw (6,4) -- (8,1);
					\draw (7,4) -- (8,1);
					\draw (8,4) -- (8,1);
					\draw (-2.5, 1) ellipse (2 and 1);
					\draw (-4,1) node {$\bullet$};
					\draw (-3,1) node {$\bullet$};
					\draw (-2,1) node {$\bullet$};
					\draw (-1.5,1) node {...};
					\draw (-0.3,1.6) node {$B_{0}^0$};
					\draw (-4,1) -- (-3,1);
					\draw (-3,1) -- (-2,1);
					\draw (-2,1) to[out=-70,in=-70] (-4,1);
					
					\draw (2.5, 1) ellipse (2 and 1);
					\draw (1,1) node {$\bullet$};
					\draw (2,1) node {$\bullet$};
					\draw (3,1) node {$\bullet$};
					\draw (3.5,1) node {...};
					\draw (4.7,1.6) node {$B_1^0$};
					\draw (1,1) -- (2,1);
					\draw (2,1) -- (3,1);
					\draw (1,1) to[out=-70,in=-70] (3,1);
					
					\draw (7.5, 1) ellipse (2 and 1);
					\draw (6,1) node {$\bullet$};
					\draw (7,1) node {$\bullet$};
					\draw (8,1) node {$\bullet$};
					\draw (8.5,1) node {...};
					\draw (9.7,1.6) node {$B^0_2$};
					\draw (6,1) -- (7,1);
					\draw (7,1) -- (8,1);
					\draw (6,1) to[out=-70,in=-70] (8,1);
					
					\draw (2.5,-1) node {$\bullet$};
					\draw (-4,1) -- (2.5,-1);
					\draw (-3,1) -- (2.5,-1);
					\draw (-2,1) -- (2.5,-1);
					\draw (3,1) -- (2.5,-1);
					\draw (2,1) -- (2.5,-1);
					\draw (1,1) -- (2.5,-1);
					\draw (8,1) -- (2.5,-1);
					\draw (7,1) -- (2.5,-1);
					\draw (6,1) -- (2.5,-1);
					
					\draw (2.5,-1.5) node {$t$};
					\draw (10.5,4) node {...};
					\draw (10,1) node {...};
				\end{tikzpicture}
			\end{minipage}
			\caption{\em  Graph $G_{5}$, a connected graph. If each $A_{i}$ is finite, then $G_{5}$ is rayless.}
		\end{figure}
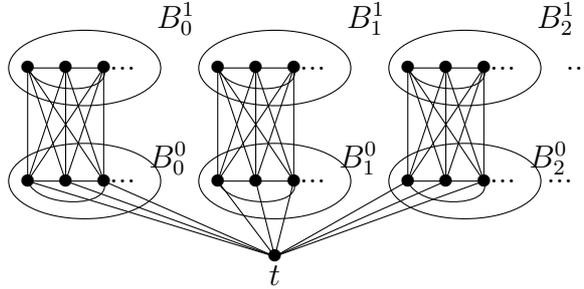
		
		\begin{flushleft}
			$V_{G_{5}} := \{t\} \cup \bigcup_{i\in I} B_i^0\cup \bigcup_{i\in I} B_i^1$,
			
			$E_{G_{5}}:=$
			$\bigg\{\{x,t\}:i\in I, x\in B^0_i\bigg\}\cup \bigg\{\{x,y\}:i\in I,x\in B^0_i, y\in B^1_i\bigg\}$ $\cup \bigg\{\{x,y\}:i\in I, x,y\in B^0_i, x\not=y\bigg\}\cup \bigg\{\{x,y\}:i\in I, x,y\in B^1_i, x\not=y\bigg\}$. 
		\end{flushleft}
		
		Let $D$ be a minimal dominating set of $G_{5}$. Define $M_{i}=(B^0_i\cup B^1_i) \cap D$ for every $i\in I$. 
		We claim that for every $i\in I$, $\vert M_{i}\vert=1$. 
		
		Case (i): If there exists an $i\in I$ such that 
		$M_{i}=\emptyset$, then 
		any member of $B^1_i$ is neither in $D$ nor it has a neighbour in $D$. This contradicts the fact that $D$ is a dominating set of $G_{5}$.
		
		Case (ii): If there exists an $i\in I$ such that $\vert M_{i}\vert \geq 2$, then pick $x,y\in M_{i}$. 
		
		\begin{itemize}
			\item Case (ii(a)): If $x,y\in B^0_i$, or $x,y\in B^1_i$, then $D\backslash\{x\}$ is a dominating set, which contradicts the minimality of $D$. 
			\item Case (ii(b)): If $x\in B^0_i$, and $y\in B^1_i$, then $D\backslash\{y\}$ is a dominating set, which contradicts the minimality of $D$.
			Similarly, we can obtain a contradiction if $y\in B^0_i$, and $x\in B^1_i$.
		\end{itemize}
		Let $M_{i}=\{a_{i}\}$ for every $i\in I$.
		Define,
		\begin{center}
			$g(i) =
			\begin{cases} 
				p^{1}_{i}(a_{i}) & \text{if}\, a_{i}\in B^1_i\cap D, \\
				
				p^{0}_i(a_{i})& \text{if} \, a_{i}\in B^0_i\cap D,
			\end{cases}$
		\end{center}
		
		where for $m\in \{0,1\}$, $p^{m}_{i}: B^m_i\rightarrow A_{i}$ is the projection map to the first coordinate for  each $i\in I$. Then, $g$ is a choice function for $\mathcal{A}$. 
	\end{remark}
	
	\begin{remark}
		The statement ``Any connected bipartite graph has a minimal edge cover" implies $\mathsf{AC}$.
		Assume $\mathcal{A}=\{A_{i}:i\in I\}$ as in the proof of Remark 6.1. Consider a family $\{(B_{i}, <_{i}):i\in I\}$ of well-ordered sets with fixed well-orderings such that for each $i\in I$, $B_{i}$ is disjoint from $A=\bigcup_{i\in I}A_{i}$ and the other $B_{j}$’s, and there is no mapping with domain $A_{i}$ and range $B_{i}$ (see the proofs of \cite[Theorem 1]{DM2006} and Theorem 5.1((4)$\Rightarrow$(1))).
		Let $B=\bigcup_{i\in I}B_{i}$. Then given some $t \not\in B\cup(\bigcup_{i\in I}A_{i})$, consider the following connected bipartite graph $G_{6}=(V_{G_{6}},E_{G_{6}})$:
		
		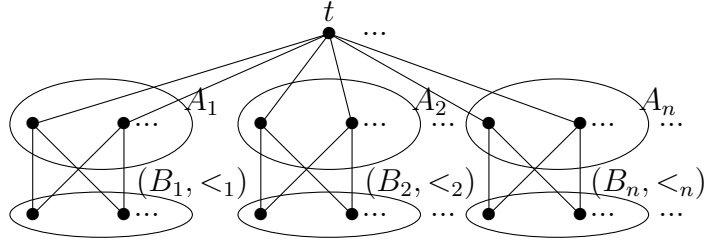
\begin{figure}[!ht]
			\centering
			\begin{minipage}{\textwidth}
				\centering
				\begin{tikzpicture}[scale=0.6]
					\draw (-2.5, 1) ellipse (2 and 1);
					\draw (-4,1) node {$\bullet$};
					\draw (-2,1) node {$\bullet$};
					\draw (-1.5,1) node {...};
					\draw (-0.3,1.5) node {$A_{1}$};
					\draw (-2.5, -1) ellipse (2 and 0.5);
					\draw (-0.5,-0.3) node {$(B_{1},<_{1})$};
					\draw (-4,-1) node {$\bullet$};
					\draw (-2,-1) node {$\bullet$};
					\draw (-1.5,-1) node {...};
					\draw (-4,1) -- (2.5,3);
					\draw (-2,1) -- (2.5,3);
					\draw (-2,1)-- (-2,-1);
					\draw (-4,1)-- (-4,-1);
					\draw (-4,1)-- (-2,-1);
					\draw (-2,1)-- (-4,-1);
					\draw (2.5, 1) ellipse (2 and 1);
					\draw (1,1) node {$\bullet$};
					\draw (3,1) node {$\bullet$};
					\draw (3.5,1) node {...};
					\draw (4.7,1.5) node {$A_{2}$};
					\draw (2.5,3) node {$\bullet$};
					\draw (3.5,3) node {...};
					\draw (2.5,3.5) node {$t$};
					\draw (2.5, -1) ellipse (2 and 0.5);
					\draw (4.5,-0.3) node {$(B_{2},<_{2})$};
					\draw (1,-1) node {$\bullet$};
					\draw (3,-1) node {$\bullet$};
					\draw (3.5,-1) node {...};
					\draw (3,1) -- (2.5,3);
					\draw (1,1) -- (2.5,3);
					\draw (3,1)-- (3,-1);
					\draw (1,1)-- (1,-1);
					\draw (3,1)-- (1,-1);
					\draw (1,1)-- (3,-1);
					\draw (5,1) node {...};
					\draw (5,-1) node {...};
					\draw (10,1) node {...};
					\draw (10,-1) node {...};
					\draw (7.5, 1) ellipse (2 and 1);
					\draw (6,1) node {$\bullet$};
					\draw (8,1) node {$\bullet$};
					\draw (8.5,1) node {...};
					\draw (9.7,1.5) node {$A_{n}$};
					\draw (7.5, -1) ellipse (2 and 0.5);
					\draw (9.5,-0.3) node {$(B_{n},<_{n})$};
					\draw (6,-1) node {$\bullet$};
					\draw (8,-1) node {$\bullet$};
					\draw (8.5,-1) node {...};
					\draw (8,1) -- (2.5,3);
					\draw (6,1) -- (2.5,3);
					\draw (8,1)-- (8,-1);
					\draw (6,1)-- (6,-1);
					\draw (8,1)-- (6,-1);
					\draw (6,1)-- (8,-1);
				\end{tikzpicture}
			\end{minipage}
			\caption{\em Graph $G_{6}$, a connected bipartite graph. If each $A_{i}$ is finite, then $G_{6}$ is rayless.}
		\end{figure}
		
		\begin{flushleft}
			$V_{G_{6}} := \{t\} \cup B\cup (\bigcup_{i\in I}A_{i})$,
			$E_{G_{6}}:=\bigg\{\{x,t\}:i\in I, x\in A_{i}\bigg\} \cup \bigg\{\{x,y\}:i\in I, x\in A_{i}, y\in B_{i}\bigg\}$. 
		\end{flushleft}
		
		The rest follows from the arguments of the implication (4)$\Rightarrow$(1) in Theorem 5.1. 
	\end{remark}

	\begin{remark}
		The statement ``Any connected bipartite graph has a maximal matching" implies $\mathsf{AC}$.
		Assume $\mathcal{A}$ as in the proof of Remark 6.1. Pick a sequence $T=\{t_{n}:n\in I\}$ disjoint from $\bigcup_{i\in I}A_{i}$, a $t\not\in \bigcup_{i\in I} A_{i}\cup T$ and consider the following connected bipartite graph $G_{7}=(V_{G_{7}},E_{G_{7}})$: 
		
		\begin{figure}[!ht]
			\centering
			\begin{minipage}{\textwidth}
				\centering
				\begin{tikzpicture}[scale=0.5]
					\draw (-2.5, 1) ellipse (2 and 1);
					\draw (-4,1) node {$\bullet$};
					\draw (-3,1) node {$\bullet$};
					\draw (-2,1) node {$\bullet$};
					\draw (-1.5,1) node {...};
					\draw (-0.3,1.6) node {$A_{0}$};
					\draw (-2.5,-1) node {$\bullet$};
					\draw (-3,-1.3) node {$t_{0}$};
					\draw (-4,1) -- (-2.5,-1);
					\draw (-3,1) -- (-2.5,-1);
					\draw (-2,1) -- (-2.5,-1);
					
					\draw (2.5, 1) ellipse (2 and 1);
					\draw (1,1) node {$\bullet$};
					\draw (2,1) node {$\bullet$};
					\draw (3,1) node {$\bullet$};
					\draw (3.5,1) node {...};
					\draw (4.7,1.6) node {$A_{1}$};
					\draw (2.5,-1) node {$\bullet$};
					\draw (3,-1.3) node {$t_{1}$};
					\draw (3,1) -- (2.5,-1);
					\draw (2,1) -- (2.5,-1);
					\draw (1,1) -- (2.5,-1);
					
					\draw (7.5, 1) ellipse (2 and 1);
					\draw (6,1) node {$\bullet$};
					\draw (7,1) node {$\bullet$};
					\draw (8,1) node {$\bullet$};
					\draw (8.5,1) node {...};
					\draw (9.7,1.6) node {$A_{2}$};
					\draw (7.5,-1) node {$\bullet$};
					\draw (8,-1.3) node {$t_{2}$};
					\draw (8,1) -- (7.5,-1);
					\draw (7,1) -- (7.5,-1);
					\draw (6,1) -- (7.5,-1);
					
					\draw (8.5,-1) node {...};
					\draw (10,1) node {...};
					
					\draw (2.5,-3.2) node {$t$};
					\draw (2.5,-2.8) node {$\bullet$};
					\draw (-2.5,-1) -- (2.5,-2.8);
					\draw (2.5,-1) -- (2.5,-2.8);
					\draw (7.5,-1) -- (2.5,-2.8);
					
				\end{tikzpicture}
			\end{minipage}
			\caption{\em Graph $G_{7}$, a connected rayless bipartite graph.}
		\end{figure}
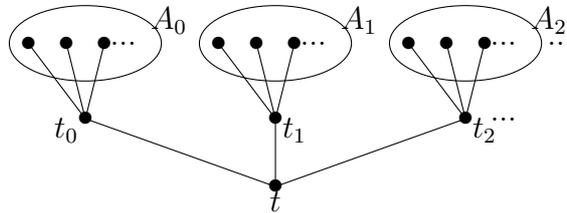
		
		\begin{flushleft}
			$V_{G_{7}}:= \bigcup_{i\in I} A_{i}\cup T \cup \{t\}$,
			$E_{G_{7}}:=\bigg\{\{t_{i},x\}: x\in A_{i}\bigg\}\cup \bigg\{\{t, t_{i}\}: i\in I\bigg\}$.    
		\end{flushleft}
		
		Let $M$ be a maximal matching of $G_{7}$. Clearly, $S=\{i\in I:\{t_{i}, t\} \in M\}$ has at most one element and for each $j\in I\backslash S$, there is exactly one $x\in A_{j}$ (say $x_{j}$) such that $\{x,t_{j}\}\in M$. 
		Let $f(A_{j})=x_{j}$ for each $j\in I\backslash S$.
		If $S\neq\emptyset$, pick any $r\in A_{i}$ if $i\in S$, since selecting an element from a set does not involve any form of choice. Let $f(A_{i})=r$. Clearly, $f$ is a choice function for $\mathcal{A}$. 
	\end{remark}
	
	\begin{thm}{($\mathsf{ZF}$)}{\em The following statements are equivalent:
			\begin{enumerate}
				\item $\mathsf{AC}$
				\item Any connected graph has a minimal dominating set.
				\item Any connected bipartite graph has a maximal matching.
				\item Any connected bipartite graph has a minimal edge cover.
			\end{enumerate} 
		}
	\end{thm}
	
	\begin{proof}
		Implications (1)$\Rightarrow$(2)-(4) are straightforward (cf. Proposition 3.3). The other directions follow from Remarks 6.1, 6.2, and 6.3. 
	\end{proof}  
	
	\begin{remark}
		The locally finite connected graphs forbid those graphs that contain vertices of infinite degrees but may contain rays. There is another class of connected graphs that forbid rays but may contain vertices of infinite degrees. For a study of some properties of the class of rayless connected graphs, the reader is referred to Halin \cite{Hal1998}. 
		
		(1). We can see that the statement ``Every connected rayless graph has a minimal dominating set" implies $\mathsf{AC_{fin}}$. Consider a non-empty family $\mathcal{A}=\{A_{i}:i\in I\}$ of pairwise disjoint finite sets and the graph $G_{5}$ from Remark 6.1. Clearly, $G_{5}$ is connected and rayless. The rest follows by the arguments of Remark 6.1.
		
		(2). By applying
		Remark 6.3 and Proposition 3.3, we can see that the statement ``Every connected rayless graph has a maximal matching" is equivalent to $\mathsf{AC}$.
		
		(3). The statement ``Every connected rayless graph has a minimal edge cover" implies $\mathsf{AC_{fin}}$. Let $\mathcal{A}=\{A_{i}:i\in I\}$ be as in (1) and $G_{6}$ be the graph from Remark 6.2. Then $G_{6}$ is connected and rayless. By the arguments of Remark 6.2, the rest follows. 
	\end{remark} 
	
	\section{Questions}
	\begin{question}
		Do the following statements imply $\mathsf{AC}$ (without assuming that the sets of colors can be well-ordered)?
		\begin{enumerate}
			\item Any graph has a chromatic index.
			\item Any graph has a distinguishing number.
			\item Any graph without a component isomorphic to $K_1$ or $K_2$ has a distinguishing index.
		\end{enumerate}
		Stawiski \cite[Theorem 3.8]{Sta2023} proved that the statements (1)--(3) mentioned above are equivalent to $\mathsf{AC}$ by assuming that the 
		sets of colors can be well-ordered.
	\end{question} 
	
	\section{Acknowledgement} The authors are very thankful to the three anonymous referees for reading the manuscript in detail and for providing several comments
	and suggestions that improved the quality and the exposition of the paper. 
	
\end{document}